\documentclass[12pt]{article}
\usepackage[margin=2.5cm, bottom=3cm]{geometry}
\usepackage[colorlinks=true, allcolors=blue]{hyperref}
\usepackage{amssymb}
\usepackage{latexsym,bm}
\usepackage{graphicx}
\usepackage{amsmath}
\usepackage{mathrsfs}
\usepackage{mathrsfs,amscd,amssymb,amsthm,amsmath,bm,graphicx,psfrag,subfigure,url,xcolor}
\usepackage{tikz}
\newtheorem{theorem}{Theorem}

\newtheorem{lemma}[theorem]{Lemma}

\newtheorem{conjecture}{Conjecture}
\newtheorem{observation}[theorem]{Observation}

\newtheorem{definition}[theorem]{Definition}

\usepackage{epstopdf}

\date{}
\newcounter{mathitem}
{\begin{list}{{$(\roman{mathitem})$}}{
\setcounter{mathitem}{0}
\usecounter{mathitem}
\setlength{\topsep}{0pt plus 2pt minus 0pt}
\setlength{\parskip}{0pt plus 2pt minus 0pt}
\setlength{\partopsep}{0pt plus 2pt minus 0pt}
\setlength{\parsep}{0pt plus 2pt minus 0pt}
\setlength{\leftmargin}{35pt}
\setlength{\itemsep}{0pt plus 2pt minus 0pt}}}
{\end{list}}
	
\begin{document}
\title{Size conditions for admissible or consecutive even cycles in graphs}
\author{Jifu Lin}
\maketitle
\footnotetext[1]{Department of Mathematics, East China Normal University, Shanghai 200241, China}
\footnotetext[2]{E-mail addresses: {jifulin01@163.com} (J. Lin).}

\begin{abstract}
	
In 2022, Gao, Huo, Liu, and Ma proved that every graph with minimum degree at least $k+1$ contains $k$ admissible cycles, where a set of $k$ cycles is said to be admissible if their lengths form an arithmetic progression with common difference one or two. In this paper, we provide
a sharp size analogue of their result and characterize the extremal graphs attaining the lower bound. In 2016, Verstra\"ete conjectured that every $n$-vertex graph $G$ containing no $k$ cycles of consecutive even lengths has at most $(2k+1)(n-1)/2$ edges, with equality only if every block of $G$ is a clique of order $2k+1$. We prove this conjecture for $2k+2\leq n\leq 4k+1$, and in fact obtain a stronger result in this range.

\end{abstract}


\vskip 8mm

\section{Introduction}

The distribution of cycle lengths in graphs has long been a significant problem in graph theory. Erdős posed numerous fundamental open problems on the distribution of cycle lengths in graphs (e.g., \cite{Erd76,Erd92,Erd95}). One such question of Erdős (see \cite{BV}) is as follows: does every simple graph with minimum degree at least three contain two cycles whose lengths differ by one or two? Bondy and Vince \cite{BV} completely settled this problem in 1998 via a stronger conclusion.

We say that $k$ cycles $C_1,C_2,\ldots,C_k$ are {\em admissible} if $|C_1|,|C_2|,\ldots,|C_k|$ form an arithmetic progression of length $k$ with common difference one or two. In this work, we aim to explore how the size of a graph (equivalently, the number of edges) influences the existence of admissible cycles or consecutive even cycles. 

In 2018, Liu and Ma \cite{12} established tight results on cycle lengths in graphs with a given minimum degree. They also proposed the following conjecture, which implies Erdős' question above when $k=2$.

\begin{conjecture}[Liu and Ma \cite{12}]\label{c2}
	Every graph with minimum degree at least $k+1$ contains $k$ admissible cycles
\end{conjecture}

In their work, Liu and Ma \cite{12} proved Conjecture \ref{c2} for bipartite graphs, and the general case under the relaxed minimum degree condition of $k+4$.

	Conjecture \ref{c2} was later completely resolved by Gao, Huo, Liu and Ma \cite{10}, who in fact settled several classical conjectures via a unified approach. This result is sharp, with sharpness confirmed by the clique $K_{k+2}$ and the complete bipartite graph $K_{k+1,n}$, and the corresponding statement is given as follows.

\begin{theorem}[Gao, Huo, Liu and Ma~\cite{10}]\label{ta}
	Every graph $G$ with minimum degree at least $k + 1$ contains $k$ admissible cycles.
\end{theorem}


In light of Theorem \ref{ta}, it is natural to consider the corresponding size problem. In what follows, we establish our main result, providing a sharp size analogue of Theorem \ref{ta} together with a characterization of extremal graphs.

Let $f(n,k)=\begin{cases}
	\binom{k+1}{2}+\binom{n-k}{2}, &\text{if }k+2\leq n\leq 2k+1,\\
	k(n-k), &\text{if } n\geq 2k+2,
\end{cases}$

$\mathcal{H}_{n,k} = \{K_1\vee (K_{k}\cup K_{n-k-1}) \mid k+2\leq n\leq 2k+1 \} \cup \{K_{(k+1)/{2}}\vee bK_1 \mid b\in\{k,k+1\}, k\geq 3 \text{ is odd}\} \cup \{K_{k,n-k} \mid n\geq 2k\}$ for $k\geq 2$, and $\mathcal{H}_{n,1}=\{T \mid T \text{ is a tree of order } n\}$.

\begin{theorem}\label{t1}
	Let $G$ be a graph of order $n\geq k+2$. If $e(G)\geq f(n,k)$, then $G$ contains $k$ admissible cycles, unless $G$ is isomorphic to	one of the graphs in $\mathcal{H}_{n,k}$.
\end{theorem}

We note that for any graph $H \in \mathcal{H}_{n,k}$, we have $e(H)=f(n,k)$ and $H$ contains no $k$ admissible cycles, which shows that Theorem \ref{t1} is sharp.

\vspace{0.2cm}

In addition, the existence of consecutive even cycles in graphs is another important problem in extremal graph theory. Häggkvist and Scott~\cite{HS} showed that every graph with minimum degree $\Omega(k^2)$ contains $k$ cycles of consecutive even lengths. Later, Verstraëte~\cite{V2} improved this quadratic bound to a linear one, proving that every graph with average degree at least $8k$ and even girth $g$ contains $(g/2 - 1)k$ cycles of consecutive even lengths. In subsequent work, Sudakov and Verstraëte~\cite{SV} further strengthened the result by establishing an exponential bound: every graph with average degree at least $192(k+1)$ and girth $g$ contains $k^{\lfloor (g-1)/2 \rfloor}$ cycles of consecutive even lengths. Verstra\"ete~\cite[Theorem 12]{V1} proved that every graph with at least $3kn$ edges contains $k$ cycles with consecutive even lengths. 

Recently, Liu and Ma \cite{12} established the minimum degree condition for the existence of the $\lfloor k/2 \rfloor$ cycles with consecutive even lengths. This minimum degree condition is best possible, as shown by the complete graph $K_{k+2}$.

\begin{theorem}[Liu and Ma \cite{12}]\label{tb}
	Every graph $G$ with minimum degree at least $k + 1$ contains $\lfloor k/2\rfloor$ cycles with consecutive even lengths.
\end{theorem}

In 2016, Verstra\"ete~\cite{V1} posed the following conjecture in his survey paper.

\begin{conjecture}[Verstra\"ete~\cite{V1}]\label{c1}
	If $G$ is an $n$-vertex graph not containing $k$ cycles of consecutive even lengths, 
	then $e(G)\leq (2k+1)(n-1)/2$, with equality only if every block of $G$ 
	is a clique of order $2k+1$.
\end{conjecture}

In Conjecture~\ref{c1}, the case $k=1$ is trivial, the case $k=2$ was proved by Gao, Li, Ma and Xie~\cite{GM}, and the remaining cases are still open.

As a byproduct of our lemma, we verify Conjecture \ref{c1} for $2k+2\leq n\leq 4k+1$, 
since $(2k+1)(n-1)/2\geq \binom{2k+1}{2}+\binom{n-2k}{2}$. In fact, we obtain a stronger result in this range as follows.

\begin{theorem}\label{t3}
	Let $G$ be a graph of order $n$ with $2k+2\leq n\leq 4k+1$. If $e(G)\geq 	\binom{2k+1}{2}+\binom{n-2k}{2}$, then $G$ contains cycles of every length $4,6,8,\dots,2k+2$, unless $G\cong K_1\vee(K_{2k}\cup K_{n-2k-1})$.
\end{theorem}

In the process of proving our main result, we also obtain an improvement on the lower bound $3kn$ in the result of Verstraëte \cite[Theorem 12]{V1}, bringing us one step closer to Conjecture \ref{c1}.
\begin{theorem}\label{t5}
	Let $G$ be a graph of order $n$ with $n\geq 2k+2$. If $e(G)\geq 2k(n-k-\frac12)$, then $G$ contains $k$ cycles of consecutive even lengths.
\end{theorem}

According to Mantel's Theorem \cite{WM}, if $e(G)\geq \lfloor n^2/4 \rfloor$, then $G$ has a cycle of length $3$, unless $G \cong K_{\lceil \frac n2 \rceil, \lfloor \frac n2 \rfloor}$. This is one of the earliest results in extremal graph theory. By applying our lemma, we obtain a theorem that serves as a strengthening of this classical result.

\begin{theorem}\label{t2}
	Let $G$ be a graph of order $n$. If $e(G)\geq \lfloor n^2/4 \rfloor$, then $G$ has cycles of every length $3,4,\dots,\lfloor n/2 \rfloor+2$, unless $G$ is isomorphic to $K_{\lceil \frac n2 \rceil, \lfloor \frac n2 \rfloor}$, $K_1\vee(K_{\lceil \frac{n-1}{2} \rceil}\cup K_{\lfloor \frac{n-1}{2} \rfloor})$ or $K_2\vee 4K_1$.
\end{theorem}

The paper is organized as follows. In Section 2, we introduce the notation and tools required for later proofs. In Section 3, we give proofs of Theorems \ref{t1}, \ref{t3}, \ref{t5} and \ref{t2}.
		
\section{Preliminaries}		
		
		Throughout this paper, we consider finite simple graphs and use standard terminology and notation from \cite{2,7}. Let $G$ be a graph with vertex set $V(G)$ and edge set $E(G)$, where $n=|V(G)|$ and $e(G)=|E(G)|$ are called the \emph{order} and \emph{size} of $G$, respectively. 
		
		Let $G_1$ and $G_2$ be vertex disjoint graphs. The \emph{union} $G_1\cup G_2$ is the graph with vertex set $V(G_1)\cup V(G_2)$ and edge set $E(G_1)\cup E(G_2)$. For any positive integer $t$, let $tG$ denote the disjoint union of $t$ copies of $G$. The \emph{join} $G_1 \vee G_2$ is derived from $G_1 \cup G_2$ by joining every vertex of $G_1$ with every vertex of $G_2$ by an edge. We denote by $K_n$ the complete graph on $n$ vertices, and by $K_{m,n}$ the complete bipartite graph with parts of sizes $m$ and $n$.
		
		We use $|G|$ to denote the order of a graph $G$. For $v \in V(G)$, let $N_G(v)$ and $\deg_G(v)$ denote the neighborhood and degree of $v$. 
		An $s$-cycle is a cycle of length $s$. 
		
		For a subset $S \subseteq V(G)$, we denote by $G[S]$ the subgraph of $G$ induced by $S$, and by $G - S$ the subgraph obtained from $G$ by deleting the vertices in $S$ and their incident edges. If $S = \{v\}$, we write $G - v$ for $G - \{v\}$.
		
		A \emph{cut-vertex} of a graph $G$ is a vertex whose deletion increases the number of connected components of $G$. A \emph{block} of $G$ is a maximal subgraph that is itself connected and contains no cut-vertex. A block $B$ is termed an \emph{end-block} of $G$ if it contains at most one cut-vertex of $G$. 
		
	We use the following transformation on the block structure of a graph.
	
	\begin{definition}[Block Leaf Transformation]\label{d1}
		Let $G$ be a connected graph with block set $\mathcal{B}=\{B_1,\dots,B_m\}$. For $B\in\mathcal{B}$, we define $\mathcal{L}(G,B)$ as a connected graph $G^*$ such that $G^*=G$ if $G$ is $2$-connected, and otherwise the block set of $G^*$ is $\mathcal{B}$ and $B$ is an end-block of $G^*$. We call $G^*$ a block leaf transformation of $G$ with respect to $B$.
	\end{definition}
	
	\begin{observation}\label{o1}
		Let $G^*=\mathcal{L}(G,B)$. Then $|G|=|G^*|$ and $e(G)=e(G^*)$.
		Moreover, $G$ and $G^*$ contain the same set of cycles. Such a transformation always exists for every connected graph $G$ and
		is in general not unique. 
	\end{observation}
		
		The \emph{circumference}, denoted by $c(G)$, is the length of a longest cycle of $G$. In 1997, Brandt \cite{3} proved the following sufficient condition for a graph to contain cycles of every length from $3$ to $c(G)$. 
		
		\begin{lemma}[Brandt \cite{3}]\label{l1}
			Every non-bipartite graph $G$ of order $n$ with more than $(n-1)^2/4 + 1$ edges contains cycles of every length $s$, where $3 \leq s \leq c(G)$.
		\end{lemma}
		
		We remark that Tang and Zhan \cite{TZ} have recently strengthened Lemma \ref{l1}.
		
		The following classical result of Erdős and Gallai \cite{PE2} gives a sharp size threshold that ensures a graph to contain a cycle longer than a given length, with the exceptional extremal graphs characterized.
		
		\begin{lemma}[Erdős and Gallai \cite{PE2}]\label{l3}
			Let $G$ be a graph of order $n$ and let $c$ be an integer. If $e(G)\geq (c-1)(n-1)/2$, then $G$ contains a cycle of length at least $c$, unless $n=q(c-2)+1$ and $G\cong K_1\vee qK_{c-2}$.
		\end{lemma}
		
		The strongest result was obtained by Kopylov \cite{GK}, who improved the Erdős--Gallai bound for $2$-connected graphs. For $n\geq c\geq 4$ and $\frac c2>a\geq 1$, define $$H_{n,c,a}=K_a\vee (K_{c-2a}\cup (n-c+a)K_1).$$ Clearly, when $a \ge 2$, $H_{n,c,a}$ is $2$-connected, has no cycle of length $c$ or longer, and let $$g(n,c,a)=e(H_{n,c,a}) = \binom{c-a}{2} + (n-c+a)a.$$		
		\begin{lemma}[Kopylov \cite{GK}]\label{Ko}
			Let $n \ge c \ge 5$ and let $s = \left\lfloor \frac{c-1}{2} \right\rfloor$. If $G$ is a $2$-connected $n$-vertex graph with
			\[
			e(G) \ge \max \{ g(n,c,2), g(n,c,s) \},
			\]
			then either $G$ has a cycle of length at least $c$, or $G = H_{n,c,2}$, or $G = H_{n,c,s}$.
		\end{lemma}

\section{Proofs of Theorems \ref{t1}, \ref{t3}, \ref{t5} and \ref{t2}}

\begin{lemma}\label{l4}
	Let $G$ be an $n$-vertex graph with $3\leq k+2\leq n\leq 2k+1$. If $$e(G)\geq 	\binom{k+1}{2}+\binom{n-k}{2},$$ then $G$ contains cycles of every length $3,4,\dots,k+2$, unless $G$ is isomorphic to one of the following graphs:
	
	{\rm (i)} $K_1 \vee (K_k \cup K_{n-k-1})$, $K_{k,k}$, or $K_{k,k+1}$;
	
	{\rm (ii)} $H_{\frac{3k+1}{2},k+2,\frac{k+1}{2}}$ or $H_{\frac{3k+3}{2},k+2,\frac{k+1}{2}}$ with odd $k\geq 3$.
\end{lemma}

\begin{proof}
	For convenience, we have $f(n,k)=	\binom{k+1}{2}+\binom{n-k}{2}=k^2+(1-n)k+\frac{n^2-n}{2}$. 
	
	For $k+2\leq n\leq 2k-1$, we have $$\begin{aligned}
		f(n,k)- \frac{n^2}{4} &=k^2+(1-n)k+\frac{n^2-n}{2}- \frac{n^2}{4}\\&\geq \left(\frac{n+1}{2}\right)^2+\frac{(1-n)(n+1)}{2}+\frac{n^2-n}{2}-\frac{n^2}{4}=\frac34>0.
	\end{aligned}$$ Thus $G$ is non-bipartite for $k+2\leq n\leq 2k-1$ by $e(G)\geq f(n,k)>\frac{n^2}{4}$.
	
	For $2k\leq n\leq 2k+1$, it is easy to check $f(n,k)=\lfloor \frac{n^2}{4} \rfloor$. If $G$ is bipartite, then $G\cong K_{\lfloor \frac{n}{2} \rfloor, \lceil \frac{n}{2} \rceil}$ by $e(G)\geq f(n,k)=\lfloor \frac{n^2}{4} \rfloor$.
	
	Next, we may assume $G$ is non-bipartite. If $n \in \{3,4\}$, the result is easily verified. If $n\geq 5$, then $e(G)\geq f(n,k)\geq \lfloor \frac{n^2}{4} \rfloor> \frac{(n-1)^2}{4}+1$. Thus by Lemma \ref{l1}, $G$ contains cycles of every length $s$, where $3\leq s\leq c(G)$. The lemma is immediate when $n\leq 4$ or $c(G) \ge k+2$, so we consider the case $n\geq 5$ and $c(G) \le k+1$ in what follows. 
	
	By $n\geq 5$, we have $k\geq 2$. If $k=2$, then $n=5$ and $e(G)=f(5,2)=6= \frac{3(n-1)}{2}$, which implies $G \cong K_1 \vee 2K_2$ by Lemma \ref{l3} and $c(G)\leq k+1=3$. Thus, we may assume that $k \geq 3$.
	
	We choose a graph $G$ of order $n$ with $e(G)\geq f(n,k)$ and $c(G)\leq k+1$ that has the maximum size. To prove this lemma, it suffices to show that $G$ is isomorphic to one of the graphs listed above, since these graphs attain size $f(n,k)$.
	
	Firstly, we claim that $G$ is connected. Suppose not, and let $G'$ be the graph obtained from $G$ by adding a single edge between two vertices in different components, leaving all other edges unchanged. Then $e(G') = e(G) + 1 \ge f(n,k)$ and $c(G') = c(G) \le k+1$, which contradicts the choice of $G$. 
	
	{\noindent\bf Case 1.} $G$ has connectivity one.
	
	In this case, we claim that $G$ contains a block of order at least $k+1$. Suppose otherwise. Let $B_{\max}$ and $B_{\min}$ be a largest and a smallest end-block of $G$, respectively, such that $B_{\max} \neq B_{\min}$. Then $|B_{\min}| \le |B_{\max}| \le k$.	Choose a non-cut-vertex $v \in B_{\min}$ of $G$. Let $G''$ be obtained from $G$ by deleting all edges between $v$ and $B_{\min} \setminus \{v\}$, and joining $v$ to every vertex of $B_{\max}$. Then $e(G'') > e(G) \geq f(n,k)$. Moreover, $c(G'') \leq k+1$ since every block in $G''$ has at most $k+1$ vertices. This contradicts the choice of $G$.
	
	Hence $G$ has a block $B$ such that $|B|\geq k+1$. Now we show $G\cong K_1\vee(K_k\cup K_{k-1})$. Since $B$ may not be an end-block, we apply the block leaf transformation to $G$ with respect to $B$, obtaining $G^* = \mathcal{L}(G, B)$ as defined in Definition~\ref{d1}. Then $B$ is an end-block of $G^*$, and $e(G^*)=e(G),\ |G|=|G^*|,\ c(G)=c(G^*)$ by Observation \ref{o1}. 
	
	If $|B| = k+1$, then $G^*$ is isomorphic to a spanning subgraph of $K_1 \vee (K_k \cup K_{n-k-1})$. By $e(G^*)=e(G)$, we have $$f(n,k)\leq e(G)=e(G^*)\leq e(K_1\vee(K_k\cup K_{n-k-1}))=f(n,k).$$ Then $e(G)=e(G^*)=e(K_1\vee(K_k\cup K_{n-k-1}))$ and $G^*$ has exactly two blocks. It follows that $G^*\cong K_1\vee(K_k\cup K_{n-k-1})$, and thus $G\cong G^*\cong K_1\vee(K_k\cup K_{n-k-1})$ by Definition \ref{d1}.
	
	If $|B|=t\geq k+2$, then $t\leq n-1$ since $G$ is not $2$-connected. Thus we have
	\begin{equation}\label{equ2}
		\begin{aligned}
			e(B)&\geq e(G^*)-e(K_{n-t+1}) \\&\geq k^2+(1-n)k+\frac{n^2-n}{2}-\frac{(n-t+1)(n-t)}{2}\\&=k^2+k+(t-k-1)n+\frac{t-t^2}{2}\\&\geq k^2+k+(t-k-1)(t+1)+\frac{t-t^2}{2}\\&=\frac{t^2}{2}+\left(\frac12-k\right)t+k^2-1.
		\end{aligned}
	\end{equation}
	
	Recall that $g(n,c,a)=e(H_{n,c,a}) = \binom{c-a}{2} + (n-c+a)a$. Then we have $g(t,k+2,2)=\binom{k}{2}+2(t-k)$ and $g(t,k+2,\lfloor \frac{k+1}{2}\rfloor)=\binom{\lceil\frac{k+1}{2}\rceil+1}{2}+\lfloor \frac{k+1}{2}\rfloor(t-\lceil\frac{k+1}{2}\rceil-1)$. 
	
	We now show that $e(B) > \max\{g(t,k+2,2),\, g(t,k+2,\lfloor \tfrac{k+1}{2} \rfloor)\}$. By (\ref{equ2}), we have $$\begin{aligned}
		e(B)-g(t,k+2,2)&\geq \frac{t^2}{2}+\left(\frac12-k\right)t+k^2-1-\frac{k(k+1)}{2}-2(t-k)\\&=\frac{t^2}{2}-\left(\frac32+k\right)t+\frac12k^2+\frac52k-1.
	\end{aligned}$$
	
	Let $h_0(t)=\frac{t^2}{2}-\left(\frac32+k\right)t+\frac12k^2+\frac52k-1$. By $\frac{2k+3}{2}\leq k+2\leq t$, we derive that $h_0(t)\geq h_0(k+2)=k-2>0$. Then $e(B)-g(t,k+2,2)\geq h_0(t)>0$.
	
	To prove $e(B)> g(t,k+2,\lfloor \frac{k+1}{2}\rfloor)$, we proceed by considering the parity of $k$. 
	
	For odd $k$, it holds that $g(t,k+2,\lfloor \frac{k+1}{2}\rfloor)=\binom{\lceil\frac{k+1}{2}\rceil+1}{2}+\lfloor \frac{k+1}{2}\rfloor(t-\lceil\frac{k+1}{2}\rceil-1)=\frac{(k+1)(4t - k - 3)}{8}$. By (\ref{equ2}), we have  $$\begin{aligned}
		e(B)-g(t,k+2,\left\lfloor \frac{k+1}{2}\right\rfloor)&\geq\frac{t^2}{2}+\left(\frac12-k\right)t+k^2-1-\frac{(k+1)(4t - k - 3)}{8}\\&=\frac{1}{2} t^2 - \frac{3k}{2} t + \frac{9k^2 + 4k - 5}{8}\\&\geq \frac12\left(\frac{3k}{2}\right)^2-\left(\frac{3k}{2}\right)^2+\frac{9k^2 + 4k - 5}{8}\\&=\frac{4k - 5}{8}>0.
	\end{aligned}$$
	
	For even $k$, it holds that $g(t,k+2,\lfloor \frac{k+1}{2}\rfloor)=\frac{4kt - k^2 - 2k + 8}{8}$. By (\ref{equ2}), we have $$\begin{aligned}
		e(B)-g(t,k+2,\left\lfloor \frac{k+1}{2}\right\rfloor)&\geq\frac{t^2}{2}+\left(\frac12-k\right)t+k^2-1-\frac{4kt - k^2 - 2k + 8}{8}\\&=\frac{1}{2}t^2 + \frac{1-3k}{2}t + \frac{9k^2 + 2k - 16}{8}\\&\geq \frac{1}{2}\left(\frac{3k-1}{2}\right)^2 + \frac{1-3k}{2}\left(\frac{3k-1}{2}\right) + \frac{9k^2 + 2k - 16}{8}\\&\geq \frac{8k - 17}{8}>0.
	\end{aligned}$$
	
	Combining the above argument, $e(B) > \max\{g(t,k+2,2),\, g(t,k+2,\lfloor \tfrac{k+1}{2} \rfloor)\}$, and thus $c(G)\geq c(B) \geq k+2$ by Lemma \ref{Ko}, a contradiction with $c(G)\leq k+1$.
	
	{\noindent\bf Case 2.} $G$ is $2$-connected.
	
	Recall that $e(G)\geq f(n,k)=k^2+(1-n)k+\frac{n^2-n}{2}$. Then we have $$\begin{aligned}
		e(G)-g(n,k+2,2)&\geq k^2+(1-n)k+\frac{n^2-n}{2}-\binom{k}{2}-2(n-k)\\&=\frac{1}{2}n^2 - \left(k + \frac{5}{2}\right)n + \frac{k^2 + 7k}{2}.
	\end{aligned}$$
	
	Let $h_1(x)=\frac{1}{2}x^2 - \left(k + \frac{5}{2}\right)x + \frac{k^2 + 7k}{2}$. Then the axis of symmetry of the function $h_1(x)$ is the line $x=k+\frac52$. Since $n$ is an integer with $n\geq k+2$, we have $h_1(n)\geq \min\{h_1(k+2),h_1(k+3)\}=k-3\geq 0$. Thus $e(G)-g(n,k+2,2)\geq 0$ with equality only if $k=3$ and $n\in \{5,6\}$. 
	
	To prove $e(G)\geq g(n,k+2,\lfloor \frac{k+1}{2}\rfloor)$, we proceed by considering the parity of $k$. 
	
	For odd $k$, it holds that $g(t,k+2,\lfloor \frac{k+1}{2}\rfloor)=\frac{(k+1)(4n - k - 3)}{8}$. Then we have $$\begin{aligned}
		e(G)- g(n,k+2,\left\lfloor \frac{k+1}{2}\right\rfloor)&\geq k^2+(1-n)k+\frac{n^2-n}{2}-\frac{(k+1)(4n - k - 3)}{8}\\&=\frac{1}{2}n^2 - \frac{3k + 2}{2}n + \frac{9k^2 + 12k + 3}{8}.
	\end{aligned}$$
	
	Let $h_2(n)=\frac{1}{2}n^2 - \frac{3k + 2}{2}n + \frac{9k^2 + 12k + 3}{8}$. Since $k$ is odd and $n$ is an integer, we have $h_2(n)\geq \min\{h(\frac{3k+1}{2}),h(\frac{3k+3}{2})\}=0$. Then $e(G)- g(n,k+2,\lfloor \frac{k+1}{2}\rfloor)\geq 0$ with equality only if $n\in\{\frac{3k+1}{2},\frac{3k+3}{2}\}$ and $k$ is odd. 
	
	For even $k$, it holds that $g(n,k+2,\lfloor \frac{k+1}{2}\rfloor)=\frac{4kn - k^2 - 2k + 8}{8}$. Then we have
	$$\begin{aligned}
		e(G)- g(n,k+2,\left\lfloor \frac{k+1}{2}\right\rfloor)&\geq k^2+(1-n)k+\frac{n^2-n}{2}-\frac{4kn - k^2 - 2k + 8}{8}\\&=\frac{1}{2}n^2 - \frac{3k+1}{2}n + \frac{9k^2 + 10k - 8}{8}\\&\geq \frac{1}{2}\left(\frac{3k+1}{2}\right)^2 - \frac{3k+1}{2}\left(\frac{3k+1}{2}\right) + \frac{9k^2 + 10k - 8}{8}\\&=\frac{4k - 9}{8}>0.
	\end{aligned}$$
	
	Combining the above argument, we obtain $e(B) \geq \max\{g(n,k+2,2),\, g(n,k+2,\lfloor \tfrac{k+1}{2} \rfloor)\}$, with equality only if $n\in\{\tfrac{3k+1}{2},\tfrac{3k+3}{2}\}$ and $k\geq 3$ is odd. Therefore, $G$ is isomorphic to $H_{\frac{3k+1}{2},k+2,\frac{k+1}{2}}$ or $H_{\frac{3k+3}{2},k+2,\frac{k+1}{2}}$ for odd $k\geq 3$ by Lemma \ref{Ko} and $c(G)\leq k+1$.
\end{proof}
\vspace{0.5cm}

{\noindent\bf Proof of Theorem \ref{t1}.} Let $G$ be a graph order $n\geq k+2$ with $$e(G)\geq f(n,k)=\begin{cases}
	\binom{k+1}{2}+\binom{n-k}{2}, &\text{if }k+2\leq n\leq 2k+1,\\
	k(n-k), &\text{if } n\geq 2k+2.
\end{cases}$$

The case $k=1$ yields $f(n,k)=n-1$, which is trivial. Hence, we assume without loss of generality that $k \geq 2$.

Suppose that $G$ contains no $k$ admissible cycles. It suffices to prove that $G$ is isomorphic to one of the graphs in $\mathcal{H}_{n,k} = \{K_1\vee (K_{k}\cup K_{n-k-1}) \mid k+2\leq n\leq 2k+1 \} \cup \{K_{(k+1)/{2}}\vee bK_1 \mid b\in\{k,k+1\}, k\geq 3 \text{ is odd}\} \cup \{K_{k,n-k} \mid n\geq 2k\}$. 

If $k+2\leq n\leq 2k+1$, then $G$ is isomorphic to one of the graphs in $\mathcal{H}_{n,k}$ by our assumption and Lemma \ref{l4}. Next, we may assume $n\geq 2k+2$, and show $G\cong K_{k,n-k}$.

By Theorem \ref{ta} and our assumption, every subgraph of $G$ contains a vertex of degree at most $k$. A straightforward induction then yields an ordering $v_1, v_2, \dots, v_n$ of $V(G)$ such that for any $2 \leq i \leq n$,
$$|N_{G}(v_i)\cap\{v_1, v_2, \dots, v_{i-1}\}|\leq k.$$ Since $e(G)\geq f(n,k)=k(n-k)$, for any $2k+1\leq i\leq n$, we have
\begin{equation}\label{equ1}
	e(G[\{v_1,v_2,\dots,v_i\}])\geq k(i-k).
\end{equation}

We prove the theorem by induction on $n$, starting with the base case $n = 2k+2$.

If $n=2k+2$, then by (\ref{equ1}), we have $e(G[\{v_1, v_2, \dots, v_{2k+1}\}])\geq k(k+1)=f(2k+1,k)$, which implies $G[\{v_1, v_2, \dots, v_{2k+1}\}]$  is isomorphic to either $K_{k,k+1}$ or $K_1 \vee 2K_k$ by Lemma \ref{l4}. Since $e(G[\{v_1, v_2, \dots, v_{2k+1}\}]) = k(k+1)$, we obtain $\deg_G(v_n) = k$. If $G[\{v_1, v_2, \dots, v_{2k+1}\}]\cong K_1\vee 2K_k$, then $G$ contains cycles of every length $3,4,\dots,k+2$ by $\deg_G(v_n)=k$, a contradiction with our assumption. Thus, $G[\{v_1, v_2, \dots, v_{2k+1}\}] \cong K_{k,k+1}$. Moreover, as $G$ contains no $k$ admissible cycles, $N_G(v_n)$ is contained in the part with $k$ vertices of the complete bipartite graph $G[\{v_1, v_2, \dots, v_{2k+1}\}]$, which implies that $G \cong K_{k,k+2}$. Otherwise, $G$ would contain cycles of every length $3,5,7,\dots,2k+1$, or $2,4,6,\dots,2k+2$, which contradicts our assumption.

If $n \geq 2k+3$, then we prove the statement by induction. Assume that the conclusion holds for $n-1$, and consider the case $n$. Let $G_{n-1} = G[\{v_1, v_2, \dots, v_{n-1}\}]$. By (\ref{equ1}), we have $e(G_{n-1}) \geq k(n-1-k)$. Clearly, $G_{n-1}$ contains no $k$ admissible cycles. Thus $G_{n-1} \cong K_{k,n-1-k}$ by the induction hypothesis. It follows that $e(G_{n-1}) = k(n-1-k)$, and therefore $\deg_G(v_n) = k$. Since $G$ contains no $k$ admissible cycles, we deduce that $N_G(v_n)$ is contained in the part with $k$ vertices of the complete bipartite graph $G_{n-1} \cong K_{k,n-1-k}$, which implies that $G \cong K_{k,n-k}$. 

This completes the proof. \hfill\ensuremath{\blacksquare} 

\vspace{0.5cm}

{\noindent\bf Proof of Theorem \ref{t3}.} The result of this theorem follows directly by substituting $2k$ for $k$ in Lemma \ref{l4}.\hfill\ensuremath{\blacksquare} 

\vspace{0.5cm}

{\noindent\bf Proof of Theorem \ref{t5}.} Let $G$ be a graph of order $n$ with $n\geq 2k+2$ and $e(G)\geq 2k(n-k-\frac12)$. Assume that $G$ contains no $k$ cycles with consecutive even lengths.

By Theorem \ref{tb} and our assumption, for every subgraph of $G$, there exists a vertex whose degree is at most $2k$. Clearly, there exists an ordering $v_1,v_2,\dots,v_n$ of the vertices of $G$ such that for any $2\leq i\leq n$,
$$|N_G(v_i)\cap\{v_1,v_2,\dots,v_{i-1}\}|\leq 2k.$$  Thus for any $2k+1\leq i \leq n$, we have $
	e(G[\{v_1,v_2,\dots,v_i\}])\geq 2k(i-k-\frac 12).
$

We prove the theorem by induction on $n$, starting with the base case $n = 2k+2$. 

If $n=2k+2$, then $e(G-v_n)\geq 2k(n-1-k-\frac 12)=\frac{2k(2k+1)}{2}$, and thus $G-v_n\cong K_{2k+1}$. Since $e(G-v_n)=2k(n-1-k-\frac 12)$ and $e(G)\geq 2k(n-k-\frac 12)$, we have $\deg_G(v_n)=2k$. Hence $G$ contains $k$ cycles of every length $4,6,\dots,2k+2$, a contradiction.

If $n\geq 2k+3$, then $G-v_n$ contains $k$ cycles with consecutive even lengths by $e(G-v_n)\geq 2k(n-1-k-\frac 12)$ and the induction hypothesis, and thus $G$ contains $k$ cycles with consecutive even lengths, a contradiction. \hfill\ensuremath{\blacksquare}

\vspace{0.5cm}

{\noindent\bf Proof of Theorem \ref{t2}.}
Let $k=\lfloor \frac n2\rfloor$ and $G$ be a graph of order $n$ with $e(G)\geq \lfloor n^2/4\rfloor$.

If $n$ is even, then $n=2k$, and thus $\binom{k+1}{2}+\binom{n-k}{2}=k^2=\lfloor n^2/4\rfloor$. By Lemma \ref{l4} and $e(G)\geq \lfloor n^2/4\rfloor=\binom{k+1}{2}+\binom{n-k}{2}$, $G$ contains cycles of every length $3,4,\dots, k+2$, unless $G$ is isomorphic to $K_1\vee(K_k\cup K_{k-1})$, $K_{k,k}$, or $H_{6,5,2}=K_2\vee 4K_1$.

If $n$ is odd, then $n=2k+1$, and thus $\binom{k+1}{2}+\binom{n-k}{2}=k^2+k=\lfloor n^2/4\rfloor$. By Lemma \ref{l4} and $e(G)\geq \lfloor n^2/4\rfloor=\binom{k+1}{2}+\binom{n-k}{2}$, $G$ contains cycles of every length $3,4,\dots, k+2$, unless $G$ is isomorphic to $K_1\vee(K_k\cup K_{k})$ or $K_{k,k+1}$.\hfill\ensuremath{\blacksquare}

\vskip 5mm
{\noindent\bf Acknowledgement.} I am grateful to Professor Xingzhi Zhan for suggesting the author to study the topic of cycle lengths.


\section*{\normalsize Declaration}
				
\noindent\textbf{Conflict~of~interest}
The author declares that he has no known competing financial interests or personal relationships that could have appeared to influence the work reported in this paper.
		
\noindent\textbf{Data~availability}
No data was used for the research described in the article.

\end{document}